\documentclass{tac}
   \usepackage{amsmath}
   \usepackage{amssymb}
   \usepackage{textcomp}
   \usepackage{enumerate}
   \usepackage{url}
\input xy
\xyoption{all}

\author{Thomas Athorne}

\address{School of Mathematics and Statistics, Sheffield University,\\
Hicks Building, Hounsfield Road, Sheffield, S3 7RH.}

\title{The Coalgebraic Structure of Cell Complexes}

\copyrightyear{2012}



\eaddress{pmp10ta@shef.ac.uk}

\date{\today}

\newcommand{\cat}[1]{\mathcal{#1}}
\newcommand{\cellcx}{\mathbf{CellCx}}
\newcommand{\strata}{\mathbf{Strata}}
\renewcommand{\bar}[1]{\overline{#1}}
\newcommand{\bound}{\partial}

\newcommand{\from}{\colon}

\newcommand{\map}[2]{\mathbf{#1}#2\mathbf{\textbf{-}Map}}
\newcommand{\toparr}{\mathbf{Top}\arr}
\newcommand{\arr}{^\mathbf{2}}
\newcommand{\iso}{\cong}
\newcommand{\N}{\mathbb{N}}

\newtheoremrm{defn}{Definition}
\newtheorem{prop}{Proposition}
\newtheorem{thm}{Theorem}
\newtheorem{lem}{Lemma}
\newtheoremrm{rem}{Remark}
\newtheorem{cor}{Corollary}

\begin{document}

\maketitle

\begin{abstract}
The relative cell complexes with respect to a generating set of cofibrations are
an important class of morphisms in any model structure. In the particular case
of the standard (algebraic) model structure on $\textbf{Top}$, we give a new
expression of these morphisms by defining a category of relative cell complexes,
which has a forgetful functor to the arrow category. This allows us to prove a
conjecture of Richard Garner: considering the algebraic weak factorisation
system given in that algebraic model structure between cofibrations and trivial
fibrations, we show that the category of relative cell complexes is equivalent
to the category of coalgebras.
\end{abstract}


\section{Introduction}

The aim of this paper is the proof of a conjecture of Richard Garner, which
describes how the recent and categorically motivated concept of an
\emph{algebraic weak factorisation system} is deeply connected, in the specific
case of topological spaces, to the more established and well understood idea of
relative cell complexes. Specifically, we will prove that the \emph{left map
structures} (coalgebra structures) of the canonical algebraic weak factorisation
system on \textbf{Top} are exactly the relative cell complexes. A relative cell
complex is a morphism with a \emph{property}: it can be expressed as a
transfinite composite of pushouts of coproducts of sphere inclusions. The left
map structures, on the other hand, specify an explicit choice of such an
expression. We will spend some of the paper making this notion of a relative
cell complex \emph{structure} precise.

The notion of an algebraic weak factorisation system (which we will write as
`awfs'---both singular and plural) was introduced by Grandis and Tholen
\cite{GrandisTholen} as a way to make the notion of a weak factorisation system
more amenable to study using the techniques of category theory. Weak
factorisation systems are fundamental to homotopy theory; they are the main
component of model category structures. However, they have some drawbacks: most
important for Grandis and Tholen was the fact that the classes of left and right
maps are not closed under colimits and limits. The definition of an awfs
involves a functorial factorisation of which the left functor is a comonad and
the right functor is a monad. The notions of left and right maps are naturally
replaced with the coalgebras and algebras, meaning that colimits and limits are
automatically available.

Another possible disadvantage of weak factorisation systems is the standard
method of constructing them: the \emph{small object argument}. This is a
transfinite induction which does the job well enough for many purposes, but
which seems at odds with what should be a very natural concept in category
theory; it relies on terminating a sequence at some arbitrary choice of ordinal
because it won't converge, it has no universal property and it cannot be
considered as an instance of any other transfinite categorical construction. In
his paper \cite{Understanding}, Richard Garner demonstrates that this can be
cured; there is a very natural variant of the small object argument for awfs
which does converge, does have a nice universal property, and can be considered
as an instance of generating a free monoid in a monoidal category. This puts the
small object argument firmly in the context of well understood categorical
algebra.

This discovery of Garner's has demonstrated the value of awfs as a much neater
structure than weak factorisation systems. Emily Riehl has since gone on to
adapt the definition of model category to that of an \emph{algebraic model
category}, using awfs in place of weak factorisation systems, and this work can
be found in \cite{AlgModel}, \cite{RiehlThesis} and \cite{MonAlg}. At the same
time Garner has been applying awfs in helping to understand higher categories,
using them to classify homomorphisms between weak $n$-categories---see
\cite{Homo}. 

Throughout the work of Riehl there is an emphasis on the right maps---the
algebras---of a given awfs. There was a good reason for this; Garner proved (in
\cite{Understanding}) a theorem that characterised the right maps in any
cofibrantly generated awfs---a right map structure on a morphism is precisely a
choice of solution for every lifting problem with a generating left map.
Unfortunately, there was no such easy description of left map structures. It was
clear that any relative cell complex had a left map structure and that any left
map was at least a retract of a relative cell complex; Garner's conjecture,
which we attack in this paper, was that the left maps are exactly the relative
cell complexes (in the specific case of \textbf{Top} with the standard awfs). 

This result will allow us to access the left maps too. What is more, the left
maps---as cell complexes---are in many ways much more accessible and
understandable than the right maps. It is hard to give an example of a right map
structure that is neither trivial nor very complex, because of the infinite
number of liftings that must be specified. But to give a simple, finite, example
of a left map structure is very easy! The left maps have a constructive flavour
that makes them, in the author's opinion, easier to work with. The result will
also establish an important link between awfs and the homotopy theory which is
already understood. The relative cell complexes are a class of maps that have
been around for a long time, and they are fundamental to model categories; in a
sense, it is important to check that they are indeed the left maps of the awfs
in order to make sure that the awfs fits properly into the existing theory.

While this paper's result is restricted to one particular awfs, the technique
used should easily generalise to many other examples. In this case, there are
potential applications to higher category theory: in the case of weak
$n$-categories, the relative cell complexes are very closely connected to the
idea of \emph{computads}. This potential for generalisation will be discussed
further in Section~\ref{sec:final}.

\paragraph{The approach.}
We will prove the theorem by first considering the existing definition of
relative cell complex, which gives a class of morphisms in \textbf{Top}. In
Section~\ref{sec:cellcx} we adjust this definition in order to obtain a category
$\cellcx$ which has a natural forgetful functor $U$ to $\toparr$. After making
sure that this category defines a sensible notion of `cell complex structure' on
a map, we exhibit a right adjoint to $U$, in Section~\ref{sec:adjunction}. This
adjunction can be thought of as a nice concrete expression of the small object
argument; the universal property of a free cell complex is exactly analogous to
the lifting property the small object argument is designed to obtain. The
smallness condition, which we expect to find somewhere, appears sooner than you
might expect---it is required for composition to be defined on cell complexes,
in Section~\ref{sec:cellcx}.

In Section~\ref{sec:comonadicity} we demonstrate that the adjunction is
comonadic, so cell complexes are coalgebras for the comonad $UK$. In
Section~\ref{sec:awfs} we will see how $UK$ is the left hand side of an awfs,
and in Section~\ref{sec:univ} we describe the universal property that $\cellcx$
satisfies. This allows us, in the remaining Section~\ref{sec:result}, to prove
that $UK$ is isomorphic as a comonad to $\mathbb{L}$, the left hand functor of
the awfs we are interested in. This proves our main result: that our notion of
cell complex structure is equivalent to the left map structures.

\paragraph{Acknowledgements.}
The author would first of all like to thank his PhD supervisor Nick Gurski, as
well as the other members of the Sheffield Category Theory Seminar for providing
a stimulating environment. He is also grateful for useful discussions he had
with both Richard Garner and Emily Riehl in their recent visits to Sheffield.
Thanks are also due to the EPSRC for funding. 

\section{Background}\label{sec:background}

Any functorial factorisation on a category $\cat{C}$ can be described as a pair
$(L,R)$ of a copointed endofunctor and a pointed endofunctor on $\cat{C}\arr$
(the category of arrows in $\cat{C}$), with the following properties: 
\begin{itemize}
	\item $L$ is domain preserving,
	\item $R$ is codomain preserving, 
	\item the functors $\text{cod}\circ L$ and $\text{dom}\circ R$ are
equal,
	\item $Rf\circ Lf=f$ for any $f$.
\end{itemize}
It is useful to give the functor $\text{cod}\circ L$, or equivalently
$\text{dom}\circ R$, a name; we will call it the \emph{central} functor of
$(L,R)$, and generally write it as $M\from\cat{C}\arr\to\cat{C}$.

In an algebraic weak factorisation system, we simply ask that $L$ be a comonad
and $R$ be a monad. This turns out to be essentially the same as making a choice
of solution for every lifting problem between a coalgebra and an algebra. We
should note that the original name was \emph{natural weak factorisation system};
we follow the name adopted by \cite{AlgModel}.

\begin{defn}
An \emph{algebraic weak factorisation system} on a category $\cat{C}$ is a pair
$(\mathbb{L},\mathbb{R})$ where $\mathbb{L}=(L,\vec{\epsilon},\vec{\delta})$ is
a comonad on $\cat{C}\arr$, $\mathbb{R}=(R,\vec{\eta},\vec{\mu})$ is a monad on
$\cat{C}\arr$, the copointed endofunctor $(L,\vec{\epsilon})$ together with the
pointed endofunctor $(R,\vec{\eta})$ make up the data of a single functorial
factorisation, and the pair satisfies the \emph{distributivity axiom}, explained
below.
\end{defn}

The final condition will ensure that the monad and comonad behave properly with
respect to one another. It follows from the monad laws that $\vec{\delta}$ must
have trivial domain component and $\vec{\mu}$ must have trivial codomain
component, so their components take the forms $(1,\delta_f)$ and $(\mu_f,1)$:
\[\xymatrix{
\bullet\ar@{=}[r]\ar[d]_{Lf}\ar@{}[dr]|{\vec{\delta}_f} & \bullet\ar[d]^{LLf} &
MRf\ar[r]^{\mu_f}\ar[d]_{RRf}\ar@{}[dr]|{\vec{\mu}_f} & f\ar[d]^{Rf} \\
Mf\ar[r]_{\delta_f} & MLf & \bullet\ar@{=}[r] & \bullet 
}\] 
for some $\delta_f$ and $\mu_f$. Then we can define a natural transformation
$\Delta\from LR\to RL$ with components given by $(\delta_f,\mu_f)$. The
distributivity axiom says that this is a \emph{distributive law} of the comonad
over the monad, meaning that it commutes with the unit, counit, multiplication
and comultiplication transformations.

The best notions of left map and right map are now given to us by the algebraic
structure. Let $\map{L}{}$ be the category of coalgebras for the comonad
$\mathbb{L}$ and let $\map{R}{}$ be the category of algebras for the monad
$\mathbb{R}$. What exactly does a \emph{left map structure} on a morphism in
$\cat{M}$ look like? As always, a coalgebra is an object $f\from A\to B$
equipped with a structure map $f\to Lf$, which appears in this case as map
$\alpha$: 
\[\xymatrix{
A\ar[r]_{Lf} & Mf\ar[r]_{Rf} & B\ar@/_1pc/[l]_{\alpha}
}\] 
a kind of `partial inverse' to $f$. The coalgebra axioms translate into very
natural properties for $\alpha$; in particular $\alpha\circ f=Lf$ and
$Rf\circ\alpha=1_B$. They also force the domain part of the structure map to be
the identity on $A$---this is why we can represent the left map structure with
just one morphism in \textbf{Top}.

In his paper \cite{Understanding} Garner introduces a revised version of the
small object argument. This allows us to take any category $\cat{I}$ over
$\cat{C}\arr$ (assuming some smallness conditions similar to those for the
original small object argument) and produce an awfs for which the category
$\cat{I}$ is naturally a subcategory of the left map category. The argument is a
transfinite iteration where a single step performed on $f\from A\to B$ involves
considering the set of commutative squares
\[\xymatrix{
  X\ar[r]\ar[d]_i & A\ar[d]^f     \\
  Y\ar[r]       & B
}\]
where $i$ is in the category $\cat{I}$, and then forming the pushout of $A$ with
many copies of each $i\in\cat{I}$, one for each of the squares. This can be
visualised as `gluing' many cells onto $A$---one for every way such a cell can
be included in $B$ via $f$. When we iterate, we are essentially adding layer
after layer of cells in this way. In Garner's small object argument, there is a
mechanism to prevent us from adding superfluous cells, and as a result the
sequence converges. We obtain a factorisation of $f\from A\to Mf\to B$, where
any cell complex in $B$ can be lifted to $Mf$.

We now consider the set of morphisms in \textbf{Top}, given by the the inclusion
$S^{n-1}\to D^{n}$ for all $n\geq0$, where $S^{-1}$ is considered to be the
empty space and $S^0$ the pair of endpoints for $D^1$. We call this $\cat{J}$
and will treat it as a discrete category over $\toparr$. Applying the small
object argument to $\cat{J}$ produces an awfs on \textbf{Top} which is arguably
the most fundamental interesting example for homotopy theory; it is very close
to the weak factorisation system between cofibrations and trivial Serre
fibrations that appears in the standard model category structure on
\textbf{Top}. For the rest of this paper we will write it as
$(\mathbb{L},\mathbb{R})$; this is the awfs for which we prove our result. 

\section{Strata}\label{sec:strata} 

The class of morphisms in \textbf{Top} which we write as $\cat{J}$-cell and call
the \emph{relative $\cat{J}$-cell complexes} is usually defined to be the
smallest class containing $\cat{J}$ which is closed under coproducts, pushouts
and transfinite composition. Starting with this notion, we seek to define a
\emph{category} of relative $\cat{J}$-cell complexes, which we will call
$\cellcx$. There will be a forgetful functor $U\from\cellcx\to\toparr$ whose
image is precisely $\cat{J}$-cell. In other words, every morphism in
$\cat{J}$-cell will have one or more cell complex structure. We'll introduce
some helpful notation at this point. If $A$ is a relative cell complex, we will
generally write $UA$ as $\bound A\to\bar{A}$, and we'll call $\bound A$ the
\emph{boundary} or \emph{base space} of $A$, and $\bar{A}$ the \emph{body} of
$A$.

In our definition, cell complexes will be formed as sequences of layers which we
call \emph{strata}. Each stratum is the pushout of a coproduct of single cells.

\begin{defn}
A \emph{stratum} $(X,S)$ consists of the following data: a topological space
$X$, a set of \emph{cells} $S$, and for each $s\in S$ a choice of
$\kappa_s\in\cat{J}$ and a continuous map $b_s\from\bound\kappa_s\to X$.
\end{defn}

The boundary of the stratum is $X$, and the body is given by the following
pushout square:
\[\xymatrix{
  \coprod_{s\in S}\bound\kappa_s \ar[rr]^{\coprod_{s\in
S}U\kappa_s}\ar[d]_{\langle b_s\rangle_{s\in S}} && \coprod_{s\in
S}\bar{\kappa_s}  \ar[d]    \\
  X \ar[rr] && \bar{(X,S)}                  \\
}\]
So we can consider a stratum as a special sort of relative cell complex, for
which $U(X,S)$ is the bottom arrow of the diagram. We will later define general
relative cell complexes as sequences of strata satisfying certain properties;
first we will consider some properties of the category of strata. To work with
the category $\strata$ we must first define the morphisms.

\begin{defn}
Let $(X,S)$ and $(Y,T)$ be any two strata. A \emph{morphism of strata}
$(f,p)\from (X,S)\to (Y,T)$ is a continuous function $f\from X\to Y$ and a
function $p\from S\to T$, satisfying the requirement that for every $s\in S$,
$\kappa_s=\kappa_{p(s)}$ and $f\circ b_s=b_{p(s)}$.
\end{defn}

Composition is the obvious thing: $(f,p)\circ(f',p')=(f\circ f',p\circ p')$. The
functor $U\from\strata\to\toparr$ is defined as one would expect; $\bound(f,p)$
is just $f$, and $\bar{(f,p)}\from\bar{(X,S)}\to\bar{(Y,T)}$ is the unique map
making the various diagrams commute.

The first thing we note about $\strata$ is that each object $J$ of $\cat{J}$ has
a canonical strata-structure, $(\bound J,\{*\})$ where $\kappa_*=J$ and
$b_*=1_{\bound J}$. Secondly, we note that strata-structures can be transferred
along pushout: if $f\from X\to Y$ has a strata-structure $(X,S)$ then in the
pushout
\[\xymatrix{
  X \ar[r]^f\ar[d]_g & Y \ar[d]    \\
  Z \ar[r] & Z \coprod_X Y                  \\
}\]
the bottom map has a canonical strata-structure, given in the following
definition.

\begin{defn}
Given any stratum $(X,S)$ and map $g\from X\to Z$, the \emph{pushforward} of
$(X,S)$ along $g$, which we write $g_*(X,S)$, is the stratum $(Z,S)$ with each
$\kappa_s$ the same as in the original stratum and each $b_s$ given by the
original $b_s$ composed with $g$.
\end{defn}

We can see that $Z\coprod_X Y$ is the body of $(Z,S)$ by commutativity of
pushouts. Next we will consider colimits in $\strata$.
(0.45)
\begin{prop}
The category $\strata$ has all small colimits, and the functor
$U\from\strata\to\toparr$ preserves them.
\end{prop}
\begin{proof}
We check coproducts and then coequalisers. Let $(X_a,S_a)$ be a set of strata
indexed by $A$. We claim that
\[\coprod_{a\in A}(X_a,S_a)=(\coprod_{a\in A}X_a,\coprod_{a\in A}S_a)\]
where each $\kappa_s$ is the same as its original on the left, and each $b_s$ is
given by composition of its original with the inclusion map. Given a set of
strata morphisms $(\lambda_a,\rho_a)\from(X_a,S_a)\to(Y,T)$ there is a unique
pair $(f,p)$ making the following diagrams commute:
\[\xymatrix{
(X_a,S_a)\ar[r]\ar[dr]_{(\lambda_a,\rho_a)} 
 & (\coprod_{a\in A}X_a,\coprod_{a\in A}S_a)\ar[d]^{(f,p)} \\
 & (Y,T)
}\]
We must merely check that this $(f,p)$ is a strata morphism; given $s\in S_a$,
$\kappa_s=\kappa_{\rho_a(s)}=\kappa_{p(s)}$, and $b_{p(s)}=\lambda_a\circ
b_s=f\circ i_a\circ b_s$ where $i_a$ is the inclusion $X_a\to\coprod X_a$. We
see that $U$ preserves this coproduct because pushouts and coproducts commute.

As for coequalisers, consider a pair of strata morphisms $(f,p)$ and $(g,q)$
between $(X,S)$ and $(Y,T)$. We claim their coequaliser is the strata formed by
the coequaliser of $f$ and $g$ and the coequaliser of $p$ and $q$, which we will
write as $(Z,U)$. Each $u\in U$ is an equivalence class of elements of $T$, all
of which must have the same $\kappa_t$, giving us $\kappa_u$. They don't
necessarily all have the same $b_t$. However, we know the equivalence relation
is generated by $p(s)\sim q(s)$. Since $b_{p(s)}=f\circ b_s$ and
$b_{q(s)}=g\circ b_s$ when we compose them with the coequaliser map $Y\to Z$ we
get a unique definition of $b_u$, which makes the map $(l,m)$ below into a
strata morphism.
\[\xymatrix{
(X,S)\ar@<1ex>[r]^{(f,p)}\ar@<-1ex>[r]_{(g,q)} &
(Y,T)\ar[r]^{(l,m)}\ar[dr]_{(h,r)} & (Z,U)\ar@{.>}[d]^{(k,s)} \\
   &                               & (A,V)
}\]

To check the coequaliser property, let $(h,r)$ coequalise $(f,p)$ and $(g,q)$.
We get a unique pair $(k,s)$ making the diagrams commute; as before we must
simply check this is a strata morphism. For $u\in U$,
$\kappa_{s(u)}=\kappa_{r(t)}=\kappa_t=\kappa_u$, using any $t$ in the
equivalence class of $u$. Also, $b_{s(u)}=h\circ b_t=k\circ l\circ b_t=k\circ
b_u$. Finally, we must check $U$ preserves this coequaliser; clearly
$\coprod\kappa_u$ is the coequaliser of the two maps
$\coprod\kappa_s\to\coprod\kappa_t$ given by $p$ and $q$, so the result follows
because coequalisers commute with pushouts.
\end{proof}

Finally we prove a useful lemma about morphisms of strata.

\begin{lem}[Pullback Lemma for Strata]\label{lem:stratapullback}
Let $(f,p)\from(X,S)\to(Y,T)$ be any strata morphism. The commutative square
defined by $U(f,p)$ is a pullback square. 
\end{lem}
\begin{proof}
We are considering the square
\[\xymatrix{
  X \ar[r]\ar[d]_f & \bar{(X,S)}\ar[d]^{\bar{(f,p)}}    \\
  Y \ar[r] & \bar{(Y,T)}                  
}\]
which we can demonstrate to be a pullback square using our understanding of
limits and colimits in \textbf{Top}.

Firstly, we know that $X$ is a subspace of $\bar{(X,S)}$. Given a point
$x\in\bar{(X,S)}$, assume it is not in $X$. Then it must be in some cell $s\in
S$: it is a point in $\bar{\kappa_s}$, and not a point of the boundary
$\bound\kappa_s$. So its image under $\bar{(f,p)}$ is in the same position in
the corresponding cell $p(s)\in T$, and hence not in $Y$. This demonstrates that
as a point set, $X$ is the pullback; since it is a subspace of $X$, and this
determines its open sets, it is also the pullback as a space.
\end{proof}

\section{Cell complexes}\label{sec:cellcx}

Now we understand the category of strata, we will move on to general cell
complexes. These are defined as infinite sequences of strata satisfying two
important properties. The first property says that the strata link together
correctly, while the second is a kind of normal form property: it says that
every cell appears in the lowest possible stratum.

\begin{defn}
An infinite sequence of strata is \emph{connected} if the boundary of each
stratum is equal to the body of the previous stratum.
\end{defn} 

\begin{defn}
An infinite sequence of strata, $(X_n,S_n)_\N$, is \emph{proper} if there is no
$s\in S_n$ for any $n$ such that $b_s$ can be factored through the boundary of a
lower stratum.
\end{defn}

\begin{defn}
A \emph{relative cell complex} is a proper connected sequence of strata. The
image under $U$ of such a relative cell complex is the transfinite composite in
\textbf{Top} of all the $U(X_n,S_n)$. Again we will talk about the
\emph{boundary} $\bound(X_n,S_n)_\N$ and the \emph{body} $\bar{(X_n,S_n)_\N}$.
\end{defn}

A relative cell complex has \emph{infinite height} if no $S_n$ is empty.
Alternatively it has \emph{height} $n$ if $S_{n+1}$ is the first empty set of
cells. Clearly (because of the required property of properness) if $S_n$ is
empty, then so is $S_m$ for all $m>n$. The \emph{trivial cell complex} on $X$ is
the unique height 0 cell complex with boundary $X$. We define a morphism of cell
complexes by similarly extending the definition for morphisms of strata.

\begin{defn}
Given any two relative cell complexes $(X_n,S_n)_\N$ and $(Y_n,T_n)_\N$, a
\emph{relative cell complex morphism} between them is a sequence of morphisms of
strata, $(f_n,p_n)_\N\from(X_n,S_n)_\N\to(Y_n,T_n)_\N$, satisfying the
\emph{coherence} condition---that $f_{n+1}=\bar{(f_n,p_n)}$ for all $n\geq0$.
The image under $U$ appears as 
\[\xymatrix{
X_0\ar[r]\ar[d]_{f_0} & X_1\ar[r]\ar[d]_{f_1} & X_2\ar[r]\ar[d]_{f_2} & \ldots
\ar[r] & \bar{(X_n,S_n)_\N}\ar[d]^{\bar{(f_n,p_n)_\N}} \\
Y_0\ar[r]             & Y_1\ar[r]             & Y_2\ar[r]             & \ldots
\ar[r] & \bar{(Y_n,T_n)_\N} 
}\]
where $\bar{(f_n,p_n)_\N}$ is the unique map that makes the diagram commute.
\end{defn}

We will write $\cellcx$ for the category whose objects are relative cell
complexes and whose morphisms are relative cell complex morphisms. $\strata$
embeds in $\cellcx$ as the subcategory of complexes with height less than or
equal to one. 

We must check a few important facts about $\cellcx$. Firstly, we want to extend
the result about colimits from $\strata$ to $\cellcx$. Secondly, we look at some
other constructions that can be made in the category. Finally, to make sure it
is a good candidate for a `category of relative cell complexes', we show that
the image of $U$ in $\toparr$ is exactly the class $\cat{J}$-cell.

\begin{prop}\label{prop:colimitscellcx}
The category $\cellcx$ has all small colimits, and the functor $U$ preserves
them.
\end{prop}
\begin{proof}
In fact, the colimits of $\cellcx$ can be computed component-wise; so given a
diagram of cell complexes the colimit is given by taking a colimit of strata for
each natural number $n$. This defines a sequence of strata, and a sequence of
strata morphisms. We must check these are connected, proper and coherent. We
will need some notation; let the diagram consist of $(X_{dn},S_{dn})_\N$ for $d$
ranging over the objects of the diagram category, let the proposed colimit be
$(Z_n,U_n)_\N$ and let the colimit cocone morphisms be
$(f_{dn},p_{dn})_\N\from(X_{dn},S_{dn})_\N\to(Z_n,U_n)_\N$. 

By connectedness, each $\bar{(X_{dn},S_{dn})}=X_{d(n+1)}$. Since $U$ preserves
the colimits of strata, $\bar{(Z_n,U_n)}$ is the colimit of the
$\bar{(X_{dn},S_{dn})}$, meaning that we can choose it to be equal to $Z_{n+1}$.
This shows our proposed colimit cell complex is connected. By the same argument
$f_{d(n+1)}=\bar{(f_{dn},p_{dn})}$; so we've also shown that the
$(f_{dn},p_{dn})_\N$ are all coherent.

Properness is harder to show; we will use the Pullback Lemma for strata. Assume
the proposed colimit is improper: let $u\in U_n$ with $b_u$ factorable through
$Z_{n-1}$ as in the diagram.
\[\xymatrix@=40pt{
\bound\kappa_u\ar@{.>}[r]\ar@/^1pc/[rr]^{b_s}\ar@/_0.5pc/[rd]_{b'_u}\ar[rrd]
^(0.7){b_u} & X_{d(n-1)}\ar[r]\ar[d]^(0.3){f_{d(n-1)}} & X_{dn}\ar[d]^{f_{dn}}\\
     & Z_{n-1}\ar[r] & Z_n
}\]
Now let $s\in S_{dn}$ be some cell with the property that $p_{dn}(s)=u$; such a
cell must exist by the definition of $U_n$ as a colimit. Thus the map $b_u$ can
also be factorised through $X_{dn}$ using $b_s$. Using the pullback square we
get a factorisation of $b_s$ through $X_{d(n-1)}$, yeilding a contradiction
because the objects in the diagram are assumed to be cell complexes, and hence
proper.

We also need to check the colimit property. Given a cocone of morphisms
$(X_{dn},S_{dn})_\N\to(Y_n,T_n)_\N$, there is a unique candidate sequence of
strata morphisms $(g_n,q_n)_\N\from(Z_n,U_n)_\N\to(Y_n,T_n)_\N$ given by each
individual colimit property in $\strata$. We just have to check this sequence is
coherent; this follows easily from the fact that each sequence of strata is
connected and that $U$ preserves colimits of strata. Finally, since $U$ of a
cell complex is defined by transfinite composition, we can see that $U$
preserves colimits using the fact that transfinite composition commutes with
other colimits. 
\end{proof}

Now recall the definition of pushforwards in the category $\strata$; this can be
extended to $\cellcx$. We construct the pushforward of each stratum in turn and
because of the connectedness property there is only one way this can be done.

\begin{defn}\label{defn:pushforward}
Given any cell complex $(X_n,S_n)_\N$ and any map $g\from X_0\to Z$, the
\emph{pushforward} of $(X_n,S_n)_\N$ along $g$, which we write as
$g_*(X_n,S_n)_\N$, is the following $(Z_n,S_n)_\N$. Firstly, $Z_0=Z$ and
$(Z_0,S_0)$ is the stratum $g_*(X_0,S_0)$. There is then a map
$\bar{(X_0,S_0)}\to\bar{(Z_0,S_0)}$ which we call $g_1$; $(Z_1,S_1)$ is defined
to be the stratum $(g_1)_*(X_1,S_1)$. Continuing in this manner we construct
each stratum of $(Z_n,S_n)_\N$, and we obtain a morphism of cell complexes where
each commutative square in the sequence is a pushout square. It is a trivial
equivalence of two colimits to see that then $Ug_*(X_n,S_n)_\N$ is the pushout
of $U(X_n,S_n)_\N$ along $g$.
\end{defn}

A new construction that we can perform in $\cellcx$, which was not possible in
$\strata$, is that of composition. Suppose we are given two cell complexes, with
the boundary of the second equal to the body of the first. Because of some
smallness conditions satisfied by the maps of $\cat{J}$ in \textbf{Top}, we can
combine the two into a single complex, whose underlying map is the composite in
\textbf{Top} of the underlying maps of the two original complexes. This
construction expresses the intuition that cell complexes can be glued onto one
another to make larger complexes.

To define the composite of two cell complexes in general, we will start with a
simple case. Let $(X_n,S_n)_\N$ be any cell complex and $(Y,T)$ be a stratum,
which we consider as a height one cell complex. Also let $Y=\bar{(X_n,S_n)_\N}$,
so that composition makes sense. We define the composite, which we will write as
$(Z_n,U_n)_\N$, as follows. First, we use the standard result (see, for example,
Proposition~2.4.2 in \cite{Hovey}) that compact spaces are finite relative to
closed $T_1$ inclusions. Each map $X_n\to X_{n+1}$ is a closed $T_1$ inclusion,
and the boundary of every cell is compact---hence for every $t\in T$ there is a
smallest $n_t$ such that $b_t$ factors through $X_{n_t}$. This partitions $T$
into a sequence of sets, $(T_n)_\N$. 

Now let $Z_0=X_0$ and let $U_0=S_0+T_0$. Thereafter, we let each
$Z_n=\bar{(Z_{n-1},U_{n-1})}$ and each $U_n=S_n+T_n$, where all the $b_s$ and
$b_t$ are defined in the obvious way. A straightforward equivalence of two
colimits shows that the underlying map of the composite is the composite of the
underlying maps. It is also important to note that there is a canonical cell
complex morphism $(X_n,S_n)_\N\to(Z_n,U_n)_\N$, with $(1_{X_0},U(Y,T))$ as its
underlying morphism in $\toparr$:
\[\xymatrix{
X_0\ar[rr]^{U(X_n,S_n)_\N}\ar@{=}[d] && Y\ar[d]^{U(Y,T)} \\
X_0\ar[rr]_-{U(Z_n,U_n)_\N}           && \bar{(Y,T)}.
}\]

\begin{defn}\label{defn:composition}
Given any two cell complexes $A$ and $B$, such that the body of $A$ is the
boundary of $B$, we define the \emph{composite} $B*A$ by repeating the above
construction for each stratum in $B$. This gives a sequence of cell complexes
$A_n$, where each $A_n$ is $A$ with the first $n$ strata of $B$ composed onto
it. Because we have all small colimits in $\cellcx$, we can define $B*A$ as the
colimit of this sequence, and because $U$ preserves colimits, this has the
correct composite as its underlying map.
\end{defn}

We note here that this definition of composite can be extended easily to
composing a transfinite sequence of cell complexes, using exactly the same
technique---any ordinal sequence of cell complexes gives an ordinal sequence of
strata which we add one by one, taking the colimit at each limit ordinal.
Another very important observation is the following:

\begin{prop}\label{prop:horizontalcomp} 
Given a pair of cell complex morphisms $\phi\from A\to A'$ and $\psi\from B\to
B'$, if $A$ and $B$ are composable, $A'$ and $B'$ are composable, and
$\bound\psi=\bar{\phi}$, then they give rise to a new cell complex morphism
$\psi*\phi\from B*A\to B'*A'$, such that $U(\psi*\phi)=(\bound\phi,\bar{\psi})$.
 We call it the \emph{horizontal} composite of the two cell complex morphisms. 
\end{prop}
\begin{proof}
Consider the case where $B$ and $B'$ are height one, write $B=(Y,T)$ and
$B'=(Y',T')$, and let $\psi=(g,q)$: we must check that $q$ respects the
partitioning of $T$ and $T'$ into $(T_n)_\N$ and $(T'_n)_\N$; that is, for each
$u\in U$, we want $n_u=n_{q(u)}$. But $n_{q(u)}\leq n_u$ follows from the fact
that $(g,q)$ is a strata morphism, and the pullback lemma ensures that $n_u\leq
n_{q(u)}$. Now an induction argument on the height of $B$ will show that
$\psi*\phi$ is well defined.  
\end{proof}

It is worth noting that we have now defined a double category whose objects are
spaces, whose vertical morphisms are continuous functions, whose horizontal
morphisms are cell complexes and whose 2-cells are cell complex morphisms. 

\begin{prop}\label{prop:image}
The image of the functor $U\from\cellcx\to\toparr$ is exactly the class of
morphisms $\cat{J}$-cell.
\end{prop}
\begin{proof}
Firstly, the definition of $U(X_n,S_n)_\N$ is as a transfinite composite of
pushouts of coproducts of elements of $\cat{J}$, so the image is certainly a
subclass of $\cat{J}$-cell. To show the opposite inclusion, since it is clear
that each element of $\cat{J}$ has a $\cellcx$ structure, we need only check
that the image of $U$ is closed under coproducts, pushouts and transfinite
composites. We have just proved that all colimits exist in $\cellcx$ and are
preserved by $U$, and we have just defined pushforwards of cell complexes. We
have also just defined composites, and as we pointed out these are easily
extended to transfinite composites.
\end{proof}

Finally, there's also a pullback lemma for cell complexes.

\begin{lem}[Pullback Lemma for Cell Complexes]\label{lem:cellcxpullback}
Given any morphism of cell complexes, its image under $U$, when viewed as a
commutative square in $\mathbf{Top}$, is a pullback square.
\end{lem}
\begin{proof}
Let $(f_n,p_n)_\N\from(X_n,S_n)_\N\to(Y_n,T_n)_\N$ be a cell complex morphism.
Any map from the one point space to $\bar{(X_n,S_n)_\N}$ must factor through
some $X_n$, because the one point space is compact. Thus, given a point in $Y_0$
and a point in $\bar{(X_n,S_n)_\N}$ with the same image in $\bar{(Y_n,T_n)_\N}$,
a finite number of applications of the pullback lemma for strata will give a
unique point in $X_0$, and this shows that as a set at least, $X_0$ is the
pullback we want it to be. But its open subsets are determined by the subspace
inclusion into $\bar{(X_n,S_n)_\N}$, and this shows that it is indeed the
pullback. 
\end{proof}

\begin{rem}
This result is really a little stronger than stated; it implies that given any
morphism of cell complexes one can remove any finite number of strata from the
beginning and the remaining square is also a pullback in \textbf{Top}---simply
because it is also a morphism of cell complexes. This fact will be vital in
Section~\ref{sec:comonadicity}.
\end{rem}

\section{The adjunction}\label{sec:adjunction}

We now examine some more properties of the category of cell complexes; they will
let us see that it is a category of left maps, and in fact one with a useful
universal property with respect to $\cat{J}$. First we show that there is a
right adjoint to $U$, which makes an adjunction that will turn out to be
comonadic. In order to construct this right adjoint $K$, we will first restrict
our attention to $\strata$, and then consider the whole of $\cellcx$. It is
worth bearing in mind that the construction in this section is very closely
analogous to the small object argument; the first proposition gives a single
step, the second proposition iterates it. 

\begin{prop}
The functor $U\from\strata\to\toparr$ has a right adjoint $K_1$.
\end{prop}
\begin{proof}
Let $f\from A\to B$ be any continuous function between topological spaces; in
other words, any object of $\toparr$. Let $S$ be the set of all morphisms to $f$
of the form
\[\xymatrix{
\bound j\ar[d]_{Uj}\ar[r] & A\ar[d]^f \\
\bar{j}\ar[r] & B
}\]
in $\toparr$, for any $j\in\cat{J}$. Notice that any element $s\in S$ comes with
a canonical choice of $\kappa_s\in\cat{J}$ and $b_s\from\bound\kappa_s\to A$.
This means that $(A,S)$ is a stratum. It also comes with a canonical morphism
$(1_A,E_1f)\from U(A,S)\to f$ in $\toparr$, whose codomain part
$E_1f\from\bar{(A,S)}\to B$ is determined by the pushout property of
$\bar{(A,S)}$, using $f$ and the codomain part of each $s\in S$. We claim that
we have just constructed $K_1f$, and that $(1_A,E_1f)$ is the counit of the
adjunction.

Suppose $(X,T)$ is any stratum and $(g,h)\from U(X,T)\to f$ a morphism of
$\toparr$. Because of the pushout definition of $\bar{(X,T)}$, the function $h$
is determined by $g$ and a morphism $h_t\from U\kappa_t\to f$ in $\toparr$ for
each $t\in T$; this is all the information that makes up $(g,h)$. But each $h_t$
gives an element $s\in S$, so this information also exactly defines a morphism
of strata, $(g,t\mapsto h_t)\from(X,T)\to(A,S)$, and factors $(g,h)$ through
$(1_A,E_1f)$. The factorisation is unique because $E_1f$ is epic; we have
demonstrated the correspondence necessary for $K_1$ to be the right adjoint of
$U$. 
\end{proof}

The construction in this proof of $E_1f$ and $UK_1f$ is a functorial
factorisation of $f$---exactly the factorisation produced by the first step of
either small object argument (Garner's or Quillen's). In the next proposition we
extend the right adjoint to $\cellcx$, iterating in precisely the manner of
Garner's small object argument. 

\begin{prop}\label{prop:adjunction}
The functor $U\from\cellcx\to\toparr$ has a right adjoint $K$.
\end{prop}
\begin{proof}
In the proof of the previous proposition, we constructed a right adjoint to
$K_1\from\toparr\to\strata$ for $U$ in the case of strata. We also defined a
functor $E_1\from\toparr\to\toparr$ which appeared in the counit of the
adjunction and which will prove rather useful. Again, consider $f\from A\to B$,
any object of $\toparr$. Apply $K_1$ to $f$ to obtain the stratum $(A,S)$ and
the function $E_1f\from\bar{(A,S)}\to B$. Then apply $K_1$ again, this time to
$E_1f$, to get another stratum whose boundary is $\bar{(A,S)}$. This also gives
another new function $E_1E_1f$, to which we apply $K_1$ in turn. Continuing this
process gives a connected sequence of strata, which would be a good candidate
for $Kf$, except for one problem: the sequence is not proper.

A similar approach does work, however (with a little bit more work) if we
deliberately omit the improper cells at each stage. Let $A_0=A$ and $S_0=S$ as
defined above; there are no improper cells in the first stratum, so $(A_0,S_0)$
is the first stratum of $Kf$. We let $A_1=\bar{(A_0,S_0)}$ and define $S_1$ to
be the set of morphisms $Uj\to E_1f$ in $\toparr$, for any $j\in\cat{J}$,
satisfying the additional condition that their boundary maps do not factor
through $A_0$. As before, $(A_1,S_1)$ is clearly a stratum and we get a new
morphism $E_2f\from\bar{(A_1,S_1)}\to B$. We can continue in this fashion,
defining $A_n$ to be $\bar{(A_{n-1},S_{n-1})}$ and $S_n$ to be the set of
morphisms $Uj\to E_nf$ whose boundary maps do not factor through $A_{n-1}$. This
produces a connected sequence of strata which is this time proper by
construction. Furthermore, there is an unique map $Ef\from\bar{(A_n,S_n)_\N}\to
B$ which commutes with all the $E_nf$.
\[\xymatrix{
A_0\ar@/_1pc/[rrrd]_f\ar[r] & A_1\ar@/_0.3pc/[rrd]_(0.3){E_1f}\ar[r] &
A_2\ar[rd]^(0.6){E_2f}\ar[r] & \ldots\ar[r] &
\bar{(A_n,S_n)_\N}\ar@/^0.5pc/[ld]^{Ef} \\
                            &     &     &     B &
}\]

Let $(X_n,T_n)_\N$ be any cell complex and $(g,h)\from U(X_n,T_n)_\N\to f$ a map
of $\toparr$. The first thing we note is that the map $h$ corresponds to an
infinite sequence of maps, one from each $X_n$; we call this $h_n\from X_n\to
B$. Now $(g,h_1)$ is a map in $\toparr$ from $U(X_0,T_0)$ to $f$, so by the
adjunction between $\strata$ and $\toparr$ we obtain a strata morphism
$(g,p_0)\from(X_0,T_0)\to(A_0,S_0)$. Say $g_0=g$, and $g_1=\bar{(g_0,p_0)}$.
Then $(g_1,h_2)$ is a map in $\toparr$ from $U(X_1,T_1)$ to $E_1f$; this induces
a strata morphism $(g_1,p_1)\from(X_1,T_1)\to(A_1,S_1)$---use the same argument
as to construct the morphism $(X_1,T_1)\to K_1E_1f$, and note that each of the
cells in the image is proper. If some cell $t\in T_1$ were to have $b_{p_1(t)}$
that could factor through $X_0$, then by the Pullback Lemma $b_t$ would factor
through $X_0$ which is impossible. 

Now we repeat the construction of $(g_1,p_1)$ to define a sequence of strata
morphisms $(g_n,p_n)_\N\from(X_n,T_n)_\N\to(A_n,S_n)_\N$ which is automatically
coherent. Since $E_nf\circ g_n=h_n$ for each $n$, we have
$Ef\circ\bar{(g_n,p_n)_\N}=h$ and we have factored $(g,h)$ through $(1_A,Ef)$.
The factorisation is unique, again because $Ef$ can be seen to be epimorphic.
This demonstrates the correspondence that makes $K$ the right adjoint of $U$.
\end{proof}

It is useful to note that the construction we made in the first paragraph of the
proof, before we insisted that the sequence be proper, is effectively Quillen's
small object argument. When we altered the construction to ensure a proper
sequence, we missed out the superfluous cells; the distinction between the
`improper' sequence construction and the proper sequence construction that
follows it is precisely the distinction between Quillen's and Garner's small
object arguments.

\section{Comonadicity}\label{sec:comonadicity}

Two lemmas will be sufficient for us to prove that the adjunction is comonadic;
we'll use the standard result known as Beck's Monadicity Theorem, which can be
found in \cite{MacLane} and many other places besides.

\begin{lem}
The category $\cellcx$ has all equalisers, and the functor $U$ preserves all
equalisers.
\end{lem}
\begin{proof}
We start by showing this result for $\strata$. Let $(f,p)$ and $(g,q)$ be two
strata morphisms from $(X,S)$ to $(Y,T)$. Now let $e\from E\to X$ be the
equaliser of $f$ and $g$ in $\textbf{Top}$, and let $r\from L\to S$ be the
equaliser of $p$ and $q$ in \textbf{Set}. We claim that
$(e,r)\from(E,L)\to(X,S)$ is the equaliser we are looking for in $\strata$.
Firstly, we must check it is actually a stratum; given $l\in L$, because
$p(r(l))=q(r(l))$ we have $f\circ b_{r(l)}=g\circ b_{r(l)}$ so there's a unique
map $\bound\kappa_{r(l)}\to E$ which we use to define $b_l$. This definition of
$\kappa_l$ and $b_l$ makes $(e,r)$ automatically a strata morphism; we must just
check the limit property. Given a stratum morphism $(h,m)\from(Z,W)\to(X,S)$
which also equalises $(f,p)$ and $(g,q)$, we get a unique pair
$(k,n)\from(Z,W)\to(E,L)$, shown by the dotted arrow:
\[\xymatrix{
(Z,W)\ar@{.>}[d]_{(k,n)}\ar[dr]^{(h,m)} & & \\
(E,L)\ar[r]_{(e,r)} & (X,S)\ar@<1ex>[r]^{(f,p)}\ar@<-1ex>[r]_{(g,q)} & (Y,T)
}\] 
This $(k,n)$ is a strata morphism: it is clear that for any $w\in W$,
$\kappa_w=\kappa_{n(w)}$ and $e\circ k\circ b_w=h\circ b_w=b_{m(w)}=e\circ
b_n(w)$, which implies $k\circ b_w=b_{n(w)}$ because $e$ is monic. Furthermore,
$U$ preserves this equaliser: consider its image under $U$---firstly, the
boundary $E$ is by definition the equaliser we want. Secondly, a point in
$\bar{(X,S)}$ has the same image under $f$ and $g$ iff it is either in $E\subset
X$ or in a cell $s\in S$ such that $p(s)=q(s)$; hence as a point set,
$\bar{(E,U)}$ is the equaliser. As a space, its topology is determined by it
being a subspace of $\bar{(X,S)}$, so we are done.

To extend to $\cellcx$, we use a very similar argument to that in Proposition
\ref{prop:colimitscellcx}; we claim the equaliser of a pair of cell complex
morphisms is given as the sequence of equalisers of strata. This sequence is
connected, and the sequence of morphisms is coherent, by exactly the same
reasoning as in Proposition \ref{prop:colimitscellcx}. To show it is proper is
in fact much easier here, because the equaliser is a \emph{subcomplex} of the
first cell complex---the equaliser map is a sequence of strata inclusions. We
also check the limit property; this follows from the same argument as in
Proposition \ref{prop:colimitscellcx}. Finally, consider the image under $U$.
Using the result for $\strata$, the boundary of each stratum in the equaliser is
the correct equaliser in $\textbf{Top}$. Then, because every point in the body
appears in the boundary of some stratum (since the one point space is compact)
the image under $U$ is correct as a function of sets. It then follows it is
correct as a continuous function between spaces, again by considering subspace
inclusions which determine its topology.
\end{proof}

\begin{lem}\label{lem:conservative}
The functor $U$ is conservative.
\end{lem}
\begin{proof}
As usual, we prove this for $\strata$ and then extend the result to $\cellcx$.
Let $(f,p)\from(X,S)\to(Y,T)$ be a strata morphism and assume that $U(f,p)$ is
an isomorphism in $\toparr$. We consider the inverse of $U(f,p)$ in $\toparr$,
which we will write $(g,h)$. The function $h$ is determined by $g$ and a
morphism $h_t\from\kappa_t\to U(X,S)$ for each $t\in T$; and $U(f,p)\circ h_t$
is the canonical inclusion of $\kappa_t$ into $U(Y,T)$. This means that each
$h_t$ makes a choice of $h'(t)\in S$ such that $p(h'(t))=t$ and $h'(p(s))=s$.
This shows that $(g,h)$ has a strata morphism structure given by $(g,h')$, and
this strata morphism is an inverse to $(f,p)$, showing that it is an
isomorphism, and hence that $U\from\strata\to\toparr$ is conservative.

Consider a morphism of cell complexes,
$(f_n,p_n)_\N\from(X_n,S_n)_\N\to(Y_n,T_n)_\N$, and assume its image under $U$
is an isomorphism. This immediately shows that $f_0$ is an isomorphism. Using
the remark following the pullback lemma for cell complexes, each $f_n$ is a
pullback of $\bar{(f_n,p_n)_\N}$, and since the pullback of an isomorphism is an
isomorphism, all the $f_n$ are isomorphisms. Now use the result on $\strata$ to
see that all the strata morphisms $(f_n,p_n)$ are individually isomorphisms;
hence, $(f_n,p_n)_\N$ is an isomorphism and we are done.
\end{proof}

\begin{cor}
Since $U$ is conservative, $\cellcx$ has and $U$ preserves all equalisers, the
dual of Beck's monadicity theorem implies that the adjunction between $U$ and
$K$ is comonadic.
\end{cor}

\section{The awfs}\label{sec:awfs}

At this stage it follows directly from a result of Garner (the dual of Theorem
4.9 in \cite{Functorial}) that, since we have a comonad $UK$ whose category of
coalgebras admits a composition law, it must appear as part of an awfs whose
monad part is given by the counit. However, in the interests of clarity we will
spend some time proving this fact explicitly.

From now on we will observe the notational convention of writing morphisms in
$\toparr$ and $\cellcx$ vertically (as squares where the source and target
morphisms are horizontal). This seems to make the diagrams clearer, and is
appropriate if one considers the double category point of view. In some of the
diagrams we will also use a notational shorthand where instead of explicitly
writing a map and its factorisation, we draw arrows going to and from the middle
of an arrow to mean morphisms to and from the central object of that arrow's
factorisation. Thus a left map will be drawn as
\[\xymatrix{A\ar[rr]_{f} & & B\ar@/_1pc/[l]_{\alpha}}\]
and the image of a morphism $(a,b)$ in $\toparr$ under the factorisation will be
drawn as
\[\xymatrix@C=1.3cm{
A\ar[d]_a\ar[rr]^f & \ar[d]^{\bar{K(a,b)}} & B\ar[d]^b \\
C\ar[rr]_g & & D.
}\]
Arrows to and from the one quarter point or three quarters point of an arrow
mean the obvious thing, where the left or right part of a factorisation has been
factorised again.

\begin{prop}\label{prop:monad}
The endofunctor $E$ on $\toparr$, which was defined in
Proposition~\ref{prop:adjunction}, is a monad. 
\end{prop}
\begin{proof}
In Definition~\ref{defn:composition} and Proposition~\ref{prop:horizontalcomp}
we showed how cell complex structures can be composed; this will provide us with
the multiplication for the monad $E$. Firstly, the unit $\vec\eta\from 1\to E$
is given by $\vec\eta_f=(UKf,1_B)$ where $f\from A\to B$. Now, $KEf$ is a cell
complex which can be composed with $Kf$; we write the composite $KEf*Kf$. There
is a morphism in $\toparr$ given by $(1_A,EEf)\from U(KEf*Kf)\to f$, hence by
the adjunction there is a cell complex morphism $\phi\from KEf*Kf\to Kf$. We
define $\mu_f$ to be the codomain part of $U\phi$, and claim that $\eta$ and
$\vec\mu=(\mu,1)$ make $E$ into a monad.

First we check that $\vec\mu$ is a natural transformation (this is clear in the
case of $\vec\eta$). Consider $(a,b)\from f\to g$ any morphism of $\toparr$; in
the diagram
\[\xymatrix@R=1.5cm@C=1cm{
A\ar[d]_a\ar[rrrr]^(0.3)f & & \ar[d]_{\bar{K(a,b)}} &
\ar[d]|{\bar{K(\bar{K(a,b)},b)}}\ar@/_0.7pc/[l]_{\mu_f} & B\ar[d]^b \\
C\ar[rrrr]_(0.3)g & & & \ar@/^0.7pc/[l]^{\mu_g} & D 
}\]
we wish to compare the two sides of the naturality square which are
$\bar{K(a,b)}\circ\mu_f$ and $\mu_g\circ\bar{K(\bar{K(a,b)},b)}$. Because these
are both the body maps of cell complex morphisms whose boundary maps are $a$, we
can use the adjunction between $U$ and $K$. It's a quick diagram chase to see
that either side when composed with $Eg$ gives $b\circ EEf$, which means, by the
adjunction, that they are equal and $\vec\mu$ is natural.  

To check the monad laws, we use a similar method: $Ef\circ\mu_f\circ UKEf =
EEf\circ UKEf = Ef$ immediately shows that $\vec\mu\circ UK\eta=1$, and the
other unit law follows from $Ef\circ\mu_f\circ\bar{K(UKf,1_B)} =
EEf\circ\bar{K(UKf,1_B)} = Ef$. Finally, to demonstrate the multiplication law,
we wish to show that the diagram
\[\xymatrix@R=1.5cm@C=0.8cm{
A\ar@{=}[d]\ar[rrrrrrrr]^(0.3)f &&&&
&&\ar[lld]_{\mu_f}&\ar@/_0.7pc/[l]_{\mu_{Ef}}\ar[ld]^{\bar{K(\mu_f,1)}}&
B\ar@{=}[d] \\
A\ar[rrrrrrrr]_(0.3)f &&&&&&\ar@/^0.7pc/[ll]^{\mu_f}&& B
}\]
commutes. By the diagram chase
\begin{align*}
	Ef\circ\mu_f\circ\mu_{Ef} & = EEEf \\
                                  & = EEf\circ\bar{K(\mu_f,1)} \\
                                  & = Ef\circ\mu_f\circ\bar{K(\mu_f,1)}
\end{align*}
and the fact that the two maps we are comparing are both the body maps of cell
complex morphisms, we can use the adjunction again and the multiplication law
holds. 
\end{proof}

\begin{prop}
The pair $(UK,E)$ is an algebraic weak factorisation system.
\end{prop}
\begin{proof}
We have seen already that $UK$ is a comonad, $E$ is a monad and that they fit
together to form a functorial factorisation system. This is almost all that is
required to make $(UK,E)$ an awfs; the only remaining thing to check is the
distributivity axiom. There is a natural transformation $\Delta\from
UKE\Rightarrow EUK$ with components given by the square
\[\xymatrix{
  \bar{Kf}\ar[rr]^{UKEf}\ar[d]_{\delta_f}  && \bar{KEf}\ar[d]^{\mu_f}         \\
  \bar{KUKf}\ar[rr]_{EUKf}  && \bar{Kf}
  }\]
where $\delta_f$ is the codomain part of the comultiplication of $UK$. This
$\Delta$ is required to be what is called a \emph{distributive law} of $UK$ over
$E$; this means it must satisfy four commutative diagrams, basically saying it
commutes with the unit, counit, multiplication and comultiplication of $UK$ and
$E$. When we translate these commutative diagrams into components in $\toparr$,
they become eight identities in $\textbf{Top}$. Upon examination, six of these
identities are immediately true: four of them from the comonad and monad laws,
and two of them simply by definition of $\mu$ and $\delta$. 

The final two identities are in fact the same, and this single identity is shown
in the diagram
\[\xymatrix@C=0.8cm@R=1.5cm{
A\ar[rrrrrrrr]^(0.15)f &&&\ar@/^0.7pc/[l]^{\mu_{UKf}}& \ar@/_1pc/[ll]_{\delta_f}
& \ar@/^1pc/[ll]^{\bar{K(\delta_f,\mu_f)}}
&\ar@/_1pc/[ll]_{\mu_f}\ar@/^0.7pc/[l]^{\delta_{Ef}}&& B.
}\]
We'll use a similar argument to those in Proposition~\ref{prop:monad}. We have
\begin{align*}
EUKf\circ\mu_{UKf}\circ\bar{K(\delta_f,\mu_f)}\circ\delta_{Ef} 
  & = EEUKf\circ\bar{K(\delta_f,\mu_f)}\circ\delta_{Ef} \\
  & = \mu_f\circ EUKEf\circ\delta_{Ef} \\
  & = \mu_f \\
  & = EUKf\circ\delta_f\circ\mu_f
\end{align*}
and the maps we are comparing appear as cell complex morphisms, so we are done.
\end{proof}

Simply knowing how $UK$ appears as the comonad part of an awfs is not enough; we
have also defined pushforward and composition structures on $\cellcx$ and we
need to check that these are compatible with the awfs $(UK,E)$. First, we note
the general definition of pushforward and composition for the left maps of any
awfs; then we will check they agree on $\cellcx$. First of all, composition of
left maps has been defined by Riehl (see \cite{AlgModel}) as follows:

\begin{defn}
Given a pair of left maps $(f,\alpha)$ and $(g,\beta)$ where $f$ and $g$ are
composable. Then $gf$ has the \emph{composite left map structure} shown by the
dotted arrows in the following diagram
\[\xymatrix@R=1.4cm@C=1.2cm{
A\ar@{=}[d]\ar[rrrr]^(0.3){gf} &&& \ar@{.>}@/_0.7pc/[l]_{\mu_{gf}} & C\ar@{=}[d]
\\
A\ar[rr]_(0.3)f &\ar@(u,d)[ur]^(0.4){M(1,g)}&
B\ar@/^0.7pc/[l]^\alpha\ar[rr]_(0.3)g &\ar@{.>}[u]|{M(M(1,g)\circ\alpha,1)}&
C\ar@{.>}@/^0.7pc/[l]^\beta
}\]
where $M$ is the central functor of the awfs. We will write the composite left
map structure as $(gf,\beta\bullet\alpha)$.
\end{defn}

It is straightforward to check that this $(\beta\bullet\alpha)$ satisfies the
coalgebra axioms; more details can be found in \cite{AlgModel}. There is also
the following natural definition of pushforward. Note that we will begin using
the notation $[a,b]\from A+B\to X$ for the unique map satisfying $[a,b]\circ
i_A=a$ and $[a,b]\circ i_B=b$, where $i_A$ and $i_B$ are the inclusion maps of
the coproduct, and similarly for maps out of pushout objects.

\begin{defn}
Given a map $f\from A\to B$ with a left map structure $\alpha$ for some awfs
$(L,R)$, and a map $g\from A\to C$. Then the pushout of $f$ along $g$, which we
write as $g_*f$, has a canonical left map structure called the
\emph{pushforward} of $\alpha$ along $g$ and written as $g_*\alpha$. It is given
by considering
\[\xymatrix@C=1.4cm{
  A\ar[rr]^(0.3){f}\ar[d]_g & \ar[d]^{M(g,f_*g)}\ar@{<-}@/^1pc/[r]^\alpha &
B\ar[d]^{f_*g} \\
  C\ar[rr]_(0.3){g_*f} & \ar@{<-}@/_1pc/[r]_{g_*\alpha} & B\coprod_A C
}\]
and specifying the structure map $g_*\alpha$ as $[Lg_*f,M(g,f_*g)\circ\alpha]$.
\end{defn}

Again, checking the coalgebra axioms is very straightforward. Now, as a rather
important sanity check before we continue, we check the two definitions of
composites and pushforwards agree in $\cellcx$.

\begin{prop}
The definition of composition in $\cellcx$ given in
Definition~\ref{defn:composition} is the same as the general definition applied
to $\cellcx$ as the category of left maps for the awfs $(UK,E)$.
\end{prop}
\begin{proof}
Given two cell complexes, considered as left maps $(f,\alpha)$ and $(g,\beta)$,
it is enough to check that $(1_A,\beta\bullet\alpha)$ is a cell complex morphism
$(g,\beta)*(f,\alpha)\to K(gf)$. Then by the adjunction between $U$ and $K$, the
fact that $E(gf)\circ(\beta\bullet\alpha)=1_C$ (which is one of the coalgebra
axioms) implies that the left map structure is the correct one. To show it is a
cell complex morphism, factorise it as
\[\xymatrix@R=0.9cm@C=1.3cm{
A\ar[rr]^f\ar@{=}[d]       && B\ar[rr]^g\ar@{=}[d]           && C\ar[d]^\beta \\
A\ar[rr]^f\ar@{=}[d]       && B\ar[rr]^{UKg}\ar[d]^\alpha    &&
\bar{Kg}\ar[dd]^{\bar{K(\bar{K(1,g)}\circ\alpha,1)}}            \\
A\ar[rr]^{UKf}\ar@{=}[d]    && \bar{Kf}\ar[d]^{\bar{K(1,g)}} && \\
A\ar[rr]^{UK(gf)}\ar@{=}[d] && \bar{K(gf)}\ar[rr]^{UKE(gf)}  &&
\bar{KE(gf)}\ar[d]^{\mu_{gf}} \\
A\ar[rrrr]^{UK(gf)}         &&&& \bar{K(gf)},
}\]
where every square is a cell complex morphism: some are images under $K$ of
morphisms in $\toparr$, others like $(1,\alpha)$ and $(1,\beta)$ are cell
complex morphisms by definition. The very bottom square is the cell complex
morphism referred to as $\phi$ in Proposition~\ref{prop:monad}. We can compose
the whole diagram together, using both vertical and horizontal composition of
cell complex morphisms (see Proposition~\ref{prop:horizontalcomp}), to
demonstrate that $(1_A,\beta\bullet\alpha)$ does have the structure of a cell
complex morphism.
\end{proof}

\begin{prop}
The definition of pushforward given in Definition~\ref{defn:pushforward} is the
same as the general definition applied to $\cellcx$ as the category of left maps
for the awfs $(UK,E)$.
\end{prop}
\begin{proof}
Let $(f,\alpha)$ be a cell complex, and write $(g_*f,g_*\alpha)$ for the
pushforward given by the general definition. Write $\beta$ for the structure map
of the cell complex given by Definition~\ref{defn:pushforward}; we need to show
that $\beta=g_*\alpha$. Firstly, it is clear that $g_*\alpha\circ
g_*f=\beta\circ g_*f$---this is one half of the neccessary identity. Also, it is
clear from the definition of $\beta$ that $(g,f_*g)$ has a cell complex morphism
structure. This means that 
\begin{align*}
g_*\alpha\circ f_*g & =\bar{K(g,f_*g)}\circ\alpha \\
                    & =\beta\circ f_*g,
\end{align*}
which is the other half of the identity.
\end{proof}

\section{The universal property}\label{sec:univ}

We now have an awfs $(UK,E)$ which we wish to compare with the awfs produced by
the small object argument. One could take the approach of directly examining the
two comonads; if you draw pictures of both it becomes clear they are essentially
the same. However, in the remainder of this paper we will use a different
approach in which we exhibit a universal property of $\cellcx$ and compare it to
the universal property established by Garner for the small object argument. The
author feels that this approach---while less efficient---is more illuminating,
as it expresses the universality of the cell complex category construction and
thus helps us to see why it is a natural way to build an awfs.

There is a canonical functor $\eta\from\cat{J}\to\cellcx$ over $\toparr$, given
by assigning each map in $\cat{J}$ its canonical height one, single-cell complex
structure. The pair of $\cellcx$ and $\eta$ is universal among functors from
$\cat{J}$ to categories of left maps over $\toparr$, with respect to the
composition preserving functors between left map categories (which are exactly
morphisms of awfs---see \cite{Understanding}). To make this work, we need to be
sure that composition in an arbitrary left map category is sufficiently well
behaved; we begin by proving two lemmas that express this.

\begin{lem}\label{lem:assoc}
For any awfs $(L,R)$, the composition rule in $\map{L}{}$ is strictly
associative.
\end{lem}
\begin{proof} 
Given three composable left maps, $(f,\alpha)$, $(g,\beta)$ and $(h,\gamma)$, we
obtain two left map structures on $hgf$ given by the two ways of composing,
namely $\gamma\bullet(\beta\bullet\alpha)$ and
$(\gamma\bullet\beta)\bullet\alpha$. The structure maps for these are both
shown, using dotted arrows, in the following diagram:
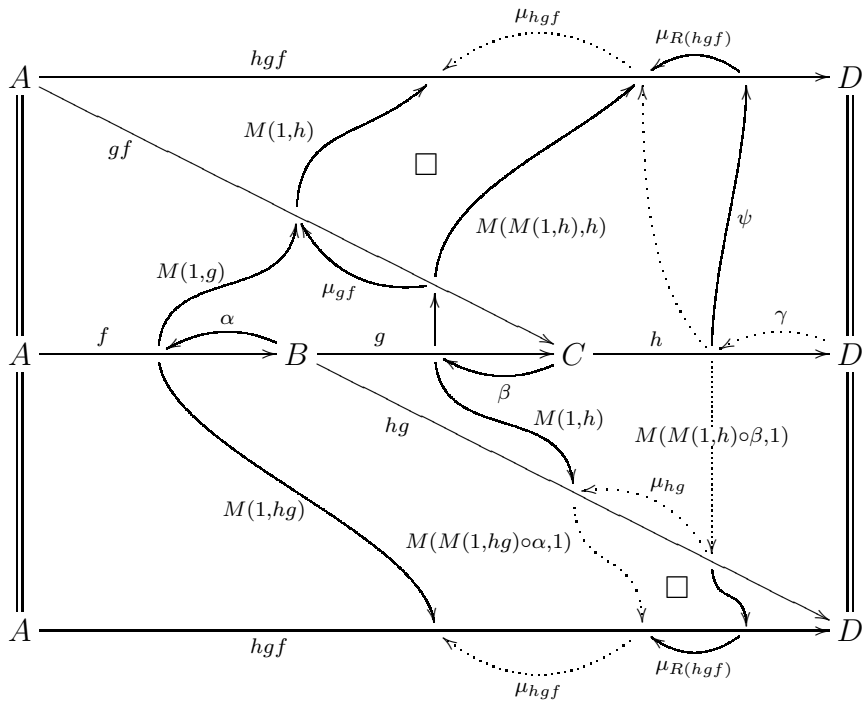
\[\xy <1.15mm,0mm>:
(0,0)*+{A}="A3"; (0,32)*+{A}="A2"; (0,64)*+{A}="A1";
(96,0)*+{D}="D3"; (96,32)*+{D}="D2"; (96,64)*+{D}="D1";
(32,32)*+{B}="B"; (64,32)*+{C}="C";
{\ar@{=} "A1"; "A2"};
{\ar@{=} "A2"; "A3"};
{\ar@{=} "D1"; "D2"};
{\ar@{=} "D2"; "D3"};
{\ar^(0.3){hgf} "A1"; "D1"};
{\ar_(0.3){hgf} "A3"; "D3"};
{\ar^(0.3){f} "A2"; "B"};
{\ar^(0.3){g} "B"; "C"};
{\ar^(0.3){h} "C"; "D2"};
{\ar_(0.2){hg} "B"; "D3"};
{\ar_(0.2){gf} "A1"; "C"};
(48,0)*i\xycircle(1,1){}="hgf2"; (72,0)*i\xycircle(1,1){}="Rhgf2";
(84,0)*i\xycircle(1,1){}="RRhgf2";
(48,64)*i\xycircle(1,1){}="hgf1"; (72,64)*i\xycircle(1,1){}="Rhgf1";
(84,64)*i\xycircle(1,1){}="RRhgf1";
(16,32)*i\xycircle(1,1){}="f"; (48,32)*i\xycircle(1,1){}="g";
(80,32)*i\xycircle(1,1){}="h";
(32,48)*i\xycircle(1,1){}="gf"; (48,40)*i\xycircle(1,1){}="Rgf";
(64,16)*i\xycircle(1,1){}="hg"; (80,8)*i\xycircle(1,1){}="Rhg";
{\ar@/_0.7pc/_{\alpha} "B"; "f"};
{\ar@/^0.7pc/^{\beta} "C"; "g"};
{\ar@{.>}@/_0.7pc/_{\gamma} "D2"; "h"};
{\ar@{.>}@/_1.4pc/_{\mu_{hgf}} "Rhgf1"; "hgf1"};
{\ar@/_0.7pc/_{\mu_{R(hgf)}} "RRhgf1"; "Rhgf1"};
{\ar@{.>}@/^1.4pc/^{\mu_{hgf}} "Rhgf2"; "hgf2"};
{\ar@/^0.7pc/^{\mu_{R(hgf)}} "RRhgf2"; "Rhgf2"};
{\ar@{.>}@/_0.9pc/_{\mu_{hg}} "Rhg"; "hg"};
{\ar@/^0.9pc/^{\mu_{gf}} "Rgf"; "gf"};
{\ar "g"; "Rgf"};
{\ar@{.>}|(0.4){M(M(1,h)\circ\beta,1)} "h"; "Rhg"};
{\ar@(u,d)^(0.4){M(1,g)} "f"; "gf"};
{\ar@(u,dl)^(0.35){M(1,h)} "gf"; "hgf1"};
{\ar@(d,u)_{M(1,hg)} "f"; "hgf2"};
{\ar@(d,u)^(0.7){M(1,h)} "g"; "hg"};
{\ar@{.>}@(d,u)_(0.2){M(M(1,hg)\circ\alpha,1)} "hg"; "Rhgf2"};
"Rhg"; "RRhgf2" **\crv{(80,2)&(84,6)}?(0.97)*\dir{>};
{\ar@(u,dl)_(0.3){M(M(1,h),h)} "Rgf"; "Rhgf1"};
{\ar@{.>}@(ul,d) "h"; "Rhgf1"};
{\ar@(u,d)_{\psi} "h"; "RRhgf1"};
(47,54)*{\Box}; (76,5)*{\Box};
\endxy \]

Using two naturality squares for $\mu$ (marked by the little squares in the
diagram) and the multiplication law, we can factor both structure maps through
$MRR(hgf)$, and hence reduce the problem to that of comparing the map
\[M(M(M(1,h),h)\circ M(M(1,g)\circ\alpha,1)\circ\beta,1)\] (which is marked in
the diagram as $\psi$) with the composite \[M(M(M(1,hg)\circ\alpha,1),1)\circ
M(M(1,h)\circ\beta,1).\] By the functoriality of $M$, this reduces to
considering
\begin{align*}
 M(M(1,h),h)\circ M(M(1,g)\circ\alpha,1) &= M(M(1,hg)\circ\alpha,h) \\
                                         &= M(M(1,hg)\circ\alpha,1)\circ
M(M(1,h))
\end{align*}
and the two dotted composites in the diagram are the same.
\end{proof}

The second lemma will prove what we will call the \emph{stacking} property of
left maps; it is absolutely vital in what follows because it justifies the
requirement for cell complexes to be proper sequences. Stacking allows you to
take a left map which is defined as a composite of colimits and move the
individual elements of the colimits about without altering the left map
structure. Since a cell complex is essentially defined as a composite of
colimits, this says we can move cells in between strata freely; hence every
potential cell complex can be reordered to make it proper---and properness
defines a natural normal form for cell complexes which hugely simplifies the
definition. 

\begin{lem}\label{lem:stacking}
For any awfs $(L,R)$, the composition rule in $\map{L}{}$ is well behaved with
respect to coproducts and pushforwards in the following way: given $f\from A\to
A'$ and $g\from B\to B'$ equipped with left map structures $\alpha$ and $\beta$,
and maps $a\from A\to X$ and $b\from B\to X$, there is an isomorphism of left
maps
\[([a,b]_*(f+g),[a,b]_*(\alpha+\beta))\iso(((a_*f)\circ b)_*g,((a_*f)\circ
b)_*\beta)\bullet(a_*f,a_*\alpha).\] 
\end{lem}
\begin{rem}
The proposition basically says that $f$ and $g$ can be `glued on' to $X$ in any
order, or simultaneously by taking a coproduct first, and it makes no difference
to the resulting left map. In the form of a picture:
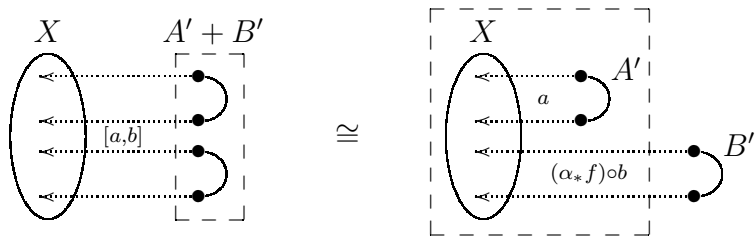
\[\xy <1mm,0mm>:
(0,0)*\xycircle(5,11){-};
(20,2)*{\bullet}="b1"; (20,8)*{\bullet}="b2" **\crv{(25,2) & (25,8)};
(20,-8)*{\bullet}="b3"; (20,-2)*{\bullet}="b4" **\crv{(25,-8) & (25,-2)};
(17,11)*{}; (26,-11)*{} **\frm{--};
{\ar@{<.} (-1,2)*{}; "b1"};
{\ar@{<.} (-1,-8)*{}; "b3"};
{\ar@{<.} (-1,8)*{}; "b2"};
{\ar@{<.} (-1,-2)*{}; "b4"};
(0,14)*{X}; (22,14)*{A'+B'}; (10,0)*{\scriptstyle{[a,b]}}
\endxy
\quad\quad \iso \quad\quad
\xy <1mm,0mm>:
(0,0)*\xycircle(5,11){-};
(13,2)*{\bullet}="b1"; (13,8)*{\bullet}="b2" **\crv{(18,2) & (18,8)};
(28,-8)*{\bullet}="b3"; (28,-2)*{\bullet}="b4" **\crv{(33,-8) & (33,-2)};
(-7,17)*{}; (22,-13)*{} **\frm{--};
{\ar@{<.} (-1,2)*{}; "b1"};
{\ar@{<.} (-1,-8)*{}; "b3"};
{\ar@{<.} (-1,8)*{}; "b2"};
{\ar@{<.} (-1,-2)*{}; "b4"};
(0,14)*{X}; (19,9)*{A'}; (34,-1)*{B'};
(8,5)*{\scriptstyle{a}}; (14,-5)*{\scriptstyle{(\alpha_*f)\circ b}};
\endxy\]
\end{rem}
\begin{proof}
The left map on the left hand side is constructed using the pushout square
\[\xy <1.3mm,0mm>:
(0,0)*+{X}="BL";
(20,0)*+{Y}="BR";
(0,10)*+{A+B}="TL";
(20,10)*+{A'+B'}="TR";
{\ar_{[a,b]_*(f+g)} "BL"; "BR"};
{\ar_{[a,b]} "TL"; "BL"};
{\ar^{[a,b]'} "TR"; "BR"};
{\ar^{f+g} "TL"; "TR"};
(16,4)*{}; (18,4)*{} **\dir{-};
(16,2)*{}; (16,4)*{} **\dir{-};
\endxy\]
while the left map on the right hand side is constructed using the two pushout
squares in the diagram
\[\xy <1.3mm,0mm>:
(0,10)*+{A}="A";
(20,10)*+{A'}="Ap";
(0,0)*+{X}="X";
(20,0)*+{Z}="Z";
(40,0)*+{W}="W";
(20,-10)*+{B}="B";
(40,-10)*+{B'}="Bp";
{\ar^f "A"; "Ap"};
{\ar_a "A"; "X"};
{\ar^{a'} "Ap"; "Z"};
{\ar^{a_*f} "X"; "Z"};
{\ar^b "B"; "X"};
{\ar_{(a_*f)\circ b} "B"; "Z"};
{\ar^{((a_*f)\circ b)_*g} "Z"; "W"};
{\ar_g "B"; "Bp"};
{\ar_{b'} "Bp"; "W"};
(16,4)*{}; (18,4)*{} **\dir{-};
(16,2)*{}; (16,4)*{} **\dir{-};
(36,-4)*{}; (38,-4)*{} **\dir{-};
(36,-2)*{}; (36,-4)*{} **\dir{-};
\endxy\]
and then composing. 

The first thing to note is that $Y$ and $W$ have exactly the same universal
property; we are thus able to choose pushouts in such a way that $Y=W$.
Furthermore, if we make this choice, the underlying maps of the left maps we are
comparing are identical. So if we check that the structure maps are equal, with
this choice of pushout objects, then for any other choice the left maps we
obtain will be isomorphic. We will now write $h\from X\to W$ for the underlying
map; we have two structure maps $W\to Mh$ and we wish to show they are equal.

Using the universal property of $W$, it is sufficient to show the maps $X\to
Mh$, $A'\to Mh$ and $B'\to Mh$ that make up these structure maps agree. The
$X\to Mh$ parts are both just $Lh$, so they are easy. Showing that the other two
parts agree can be done with two simple diagram chases. First, to simplify
notation, we will start writing $g'$ for $((a_*f)\circ b)_*g$, and $f'$ for
$a_*f$; we will also write $\alpha'$ for the pushforward structure map on $f'$
and $\beta'$ for the one on $g'$. Next, we note that the map $[a,b]'$ can be
written as $[g'a',b']$, using the universal property of $W$. Now consider the
diagram
\[\xymatrix@R=1.4cm{
 & X\ar@{=}[d]\ar[rrrr]^(0.3)h &&&\ar@/^0.7pc/[l]^{\mu_h}& W\ar@{=}[d] \\
 & X\ar[rr]^(0.3){f'} &\ar@(u,d)[ur]^(0.3){M(1,g')}&
Z\ar@/_0.7pc/[l]_{\alpha'}\ar[rr]^(0.3){g'} &\ar[u]|{M(M(1,g')\circ\alpha',1)}&
W\ar@/_0.7pc/[l]_{\beta'} \\
A\ar[ur]^{a}\ar[rr]_(0.3)f &\ar[ur]& A'\ar[ur]_{a'}\ar@/^0.7pc/[l]^\alpha &&
B\ar[ul]^{f'\circ b}\ar[rr]_(0.3)g &\ar[ul]&
B'.\ar[ul]_{b'}\ar@/^0.7pc/[l]^\beta
}\]

Two straightforward chases show that $(\beta'\bullet\alpha')\circ g'\circ
a'=M(a,g'a')\circ\alpha$ and that $(\beta'\bullet\alpha')\circ
b'=M(b,b')\circ\beta$. We then factor $M(a,g'a')$ through $M(i_A,i'_A)$ and
factor $M(b,b')$ through $M(i_B,i'_B)$, the two inclusion maps to the coproduct.
Then using the fact that, by definition of $(\alpha,\beta)$,
$M(i_A,i_A')\circ\alpha=(\alpha+\beta)\circ i'_A$ and
$M(i_B,i'_B)\circ\beta=(\alpha,\beta)\circ i'_B$, it is clear that
$(\beta'\bullet\alpha')$ and $[a,b]_*(\alpha+\beta)$ agree on both $A'$ and
$B'$.
\end{proof}

The universal property of $\cellcx$ with respect to $\cat{J}$ now follows
without too much difficulty.

\begin{prop}
For any awfs $(L,R)$ and functor $F\from\cat{J}\to\map{L}{}$ over $\toparr$,
there is a unique $F'\from\cellcx\to\map{L}{}$ over $\toparr$ which satisfies
$F=F'\circ\eta$ and preserves composition. 
\end{prop}
\begin{proof}
What can we say about such a map? Firstly, it automatically preserves colimits;
this follows from a standard argument, based on the fact that the forgetful
functors are conservative and preserve colimits themselves. Secondly, we claim
that it automatically preserves pushforwards, meaning that for any cell complex
$(f,\alpha)$ and appropriate $g$, we have $F'(g_*f,g_*\alpha)=g_*F'(f,\alpha)$.
Since $(g,f_*g)$ is a the underlying map of a cell complex morphism, by
functoriality of $F'$ it is a morphism of coalgebras for $L$. This means that
$F'(g_*\alpha)\circ f_*g=M(g,f_*g)\circ F'\alpha$, and we know the latter is
$g_*(F'\alpha)\circ f_*g$; this shows the two morphisms from the codomain of $f$
to $M(g_*f)$ are equal. It is trivial to show the two morphisms from the
codomain of $g$ to $M(g_*f)$ are equal, so by the pushout property, the two left
map structures are equal.

These properties allow us to see that for any stratum $(X,S)$, the image
$F'(X,S)$ is determined entirely, since a stratum is just a pushforward of a
coproduct of objects of $\cat{J}$, and $F'$ must take the objects of $\cat{J}$,
considered as cell complexes, to their images under $F$. Furthermore, any cell
complex is the composite of all its strata; thus if $F'$ is to preserve
composition it will be determined completely by $F$. So we have essentially
constructed a single possible candidate $F'$; now we must check that it
preserves all composites, not just the ones given by proper connected sequences
of strata. But using Lemma~\ref{lem:assoc} and Lemma~\ref{lem:stacking}, we can
take any composite of cell complexes and move the individual cells between
strata without effecting the image of the composite under $F'$; therefore $F'$
does indeed preserve composition and we are done.
\end{proof}

\section{The main result}\label{sec:result}

We are writing $(\mathbb{L},\mathbb{R})$ for the awfs that is generated from
$\cat{J}$ using Richard Garner's small object argument. This is the object we
really care about; we want to understand the coalgebras of $\mathbb{L}$. This
awfs has a certain universal property with respect to $\cat{J}$:

\begin{defn}
Given a small category $\cat{I}$ over $\cat{C}\arr$, an awfs $(L,R)$ on
$\cat{C}$ is \emph{free} with respect to $\cat{I}$ if there is a morphism
$\eta\from\cat{I}\to\map{L}{}$ over $\cat{C}\arr$ such that for any other awfs
$(L',R')$ on $\cat{C}$ and functor $F\from\cat{I}\to\map{L}{'}$ over
$\cat{C}\arr$, there is a unique awfs morphism $\alpha\from(L,R)\to(L',R')$ such
that $F=\alpha_*\circ\eta$. (The functor $\alpha_*\from\map{L}{}\to\map{L}{'}$
is the lifting of $\alpha$, as a comonad morphism, to the categories of
coalgebras.)
\end{defn} 

In his paper \cite{Understanding}, Garner both constructs
$(\mathbb{L},\mathbb{R})$ and proves that it is indeed free with respect to
$\cat{J}$. In this section we will show the same of our awfs $(UK,E)$; then,
since it is a universal property, the two awfs will be shown to be isomorphic.
We have already done the hard work---by the following lemma, the universal
property of $\cellcx$ implies that $(UK,E)$ is free with respect to $\cat{J}$.
We note that this lemma is a special case of Lemma~6.9 in \cite{AlgModel}.

\begin{lem}\label{lem:bijection}
Given two awfs, $(L,R)$ and $(L',R')$, awfs morphisms $(L,R)\to(L',R')$ are in
bijection with functors $\map{L}{}\to\map{L}{'}$ over $\cat{C}\arr$ that
preserve the composition of left maps.
\end{lem}
\begin{proof}
First, assume we have an awfs morphism $\chi\from(L,R)\to(L',R')$. In
particular, this is a comonad morphism $L\to L'$, and this means it lifts to a
functor $\tilde{\chi}\from\map{L}{}\to\map{L}{'}$ over $\cat{C}\arr$. If
$(f,\alpha)$ is an $L$-coalgebra, its image under $\tilde{\chi}$ is given by
$(f,\chi_f\circ\alpha)$---this is the standard way of lifting a comonad morphism
to the categories of coalgebras. We will check that $\tilde{\chi}$ preserves
composition of coalgebras, that is, that
$\tilde{\chi}(f,\alpha)\bullet\tilde{\chi}(g,\beta)=\tilde{\chi}\big((f,
\alpha)\bullet(g,\beta)\big)$. This is a diagram chase which proceeds as
follows:
\begin{align*}
 \chi_{gf}\circ\mu_{gf}&\circ M(M(1,g)\circ\alpha,1)\circ\beta \\
  & =\mu'_{gf}\circ M'(\chi_{gf},1)\circ\chi_{R(gf)}\circ
M(M(1,g)\circ\alpha,1)\circ\beta \\
  & =\mu'_{gf}\circ M'(\chi_{gf},1)\circ
M'(M(1,g)\circ\alpha,1)\circ\chi_g\circ\beta \\
  & =\mu'_{gf}\circ M'(M'(1,g)\circ\chi_f\circ\alpha,1)\circ\chi_g\circ\beta
\end{align*}
where the first step uses the fact that $\chi$ is a comonad morphism, the second
step uses naturality and the third step uses functoriality and naturality. 

Now we start with a functor $F\from\map{L}{}\to\map{L}{'}$ over $\cat{C}\arr$
and we assume it preserves composition. We use $F$ to define a natural
transformation $\gamma\from M\to M'L$ by writing the image of the coalgebra
$(Lf,\delta_f)$ as $(Lf,\gamma_f)$. Then we define the natural transformation
$\chi\from M\to M'$ by $\chi_f=M'(1,Rf)\circ\gamma_f$. This is the standard way
of constructing a comonad morphism from a functor between the categories of
coalgebras; hence $(1,\chi)$ is a comonad morphism $L\to L'$. We will show that
at the same time, $(\chi,1)$ is a monad morphism $R\to R'$, and that hence
$\chi$ is an awfs morphism $(L,R)\to(L',R')$. 

First, it is quick to check that $R'f\circ\chi_f=Rf$; consider
\begin{align*}
 R'f\circ M'(1,Rf)\circ\gamma_f & = Rf\circ R'Lf\circ\gamma_f \\
                               & = Rf
\end{align*}
where the first step follows from the properties of $M'(1,Rf)$ and the second
uses the fact that $\gamma_f$ is the structure map for a coalgebra. The other
identity $\chi$ must satisfy in order to be a monad morphism is shown by the
following diagram chase:
\begin{align*}
 \chi_f\circ\mu_f 
& = \chi_f\circ M(1,RRf)\circ(\delta_{Rf}\bullet\delta_f) \\
& = M'(1,RRf)\circ\chi_{(LRf\circ Lf)}\circ(\delta_{Rf}\bullet\delta_f) \\
& = M'(1,RRf)\circ F(\delta_{Rf}\bullet\delta_f) \\
& = M'(1,RRf)\circ(\gamma_{Rf}\bullet\gamma_f) \\
& = M'(1,RRf)\circ\mu'_{(LRf\circ Lf)}\circ 
       M'(M'(1,LRf)\circ\gamma_f,1)\circ\gamma_{Rf} \\
& = \mu'_f\circ M'(M'(1,RRf),RRf)\circ 
       M'(M'(1,LRf)\circ\gamma_f,1)\circ\gamma_{Rf} \\
& = \mu'_f\circ M'(\chi_f,RRf)\circ\gamma_{Rf} \\
& = \mu'_f\circ M'(\chi_f,1)\circ M'(1,RRf)\circ\gamma_{Rf} \\
& = \mu'_f\circ M'(\chi_f,1)\circ\chi_{Rf}
\end{align*}
in which we have used the fact that $\mu_f=(\delta_{Rf}\bullet\delta_f)\circ
M(1,RRf)$, the assumption that $F$ preserves composition, and the definition of
composition of $L'$-maps.

Since the correspondence we have demonstrated between awfs and composition
preserving functors is a restriction of the standard natural isomorphism between
comonad maps and functors on coalgebras, it is therefore a bijection and we are
done.
\end{proof} 

\begin{cor}
The awfs $(UK,E)$ is free with respect to $\cat{J}$.
\end{cor}

Since they share this universal property, the awfs $(UK,E)$ and
$(\mathbb{L},\mathbb{R})$ must be isomorphic. Thus we are finally able to prove
the main result of the paper.

\begin{thm}
Every map with a left-map structure with respect to $(\mathbb{L},\mathbb{R})$ is
in the class $\cat{J}$-$\mathrm{cell}$.
\end{thm}
\begin{proof}
The category $\cellcx$ which we defined is equivalent over $\toparr$ to the
category of coalgebras for the comonad $UK$, because the adjunction is
comonadic. We have also shown that $(UK,E)$ and $(\mathbb{L},\mathbb{R})$ are
isomorphic as awfs. Hence $UK$ and $\mathbb{L}$ are isomorphic as comonads; so
they have equivalent categories of coalgebras. 

Thus $\cellcx$ is equivalent over $\toparr$ to the category $\map{L}{}$ of
left-maps. Since, as we showed in Proposition~\ref{prop:image}, the image of the
functor $U$ is precisely $\cat{J}$-cell, every map with a left-map structure is
in $\cat{J}$-cell. 
\end{proof}

\section{Further thoughts}\label{sec:final}

What we have done applies only to a single awfs in a single category. Regardless
of the fact that it is arguably the most important awfs under study at the
moment, this is still quite a limitation. However, in proving this result for a
very specific case, we have demonstrated a technique that will, in the author's
view, quite easily extend to many other examples. There are many parts of the
proof where the argument would have worked equally well for any class of maps in
a category with all small limits and colimits. 

It seems likely that the construction of $\cellcx$ will work for any category
$\cat{J}$; to consider it in a more generalised context we would want to hone in
on what it \emph{really} does. It seems to consist of three steps: the first is
to freely complete $\cat{J}$ under colimits. This is a well understood
2-categorical completion which results in the category of presheaves over
$\cat{J}$. The second step is free completion under pushforwards, which produces
the category of strata. The third step is then to complete freely under
composition, resulting in the entire category of cell complexes. There is a very
close analogy between these three steps and the three steps of Garner's small
object argument which are outlined in \cite{Understanding}. 

One important feature of the definition of cell complexes given in this paper is
the fact that they are \emph{countable} sequences of strata. This is a
substantial limitation on the cardinalities involved, and one that is only
possible because of special properties enjoyed by the category $\textbf{Top}$
and the class of maps $\cat{J}$; they are the same finiteness conditions that
allow the standard small object argument to be terminated within $\omega$ steps.
In a general setting, one would need to check similar conditions, and in some
cases the ordinal number of iterations required may be larger. In these cases,
the definition of cell complex will be slightly more complicated because they
will be ordinal sequences of strata with length greater than $\omega$. However,
in the author's view, this extra complication in the definition will not be a
serious difficulty.

The largest challenges faced by an attempt to generalise this argument will come
in the case where the maps of $\cat{J}$ are not all strict monomorphisms. Many
of the definitions and propositions above are predicated quite strongly on the
assumption that `adding cells' to a complex is always a matter of `adding';
however, if a cell shape in $\cat{J}$ were not monic, the pushout we would
construct to glue it onto a complex would involve quotienting---adding a cell
could be a reduction. This may seem very counter-intuitive, but there are
important examples of awfs where exactly this behaviour would occur: in
particular, most of the awfs defined on (different strengths of) $n$-categories
have as their highest dimensional cell shape a non-monic map that identifies two
morphisms of the highest dimension. In this case, the definitions in
Section~\ref{sec:cellcx} would need to be very carefully reconsidered. For
example, if a cell shape is epimorphic, one could have any number of cells which
make no difference to the underlying morphism---most morphisms would have
infinitely many non-isomorphic cell complex structures. In particular, the
forgetful functor $U$ would not be conservative, and hence the proof would fail
as it currently stands.

\end{document}